\documentclass[12pt]{amsart}
\usepackage{geometry,multirow, amsmath, amsthm,amsfonts,amssymb,stmaryrd, mathrsfs}
\usepackage{mathdots}
\usepackage{pgf,tikz}
\usepackage{mathrsfs}
\usetikzlibrary{arrows}
\usetikzlibrary[patterns]
\usetikzlibrary{matrix,decorations.pathreplacing,calc}
\usepackage{hyperref}

\usepackage[all]{xy}

\newtheorem{theorem}{Theorem}[section]
\newtheorem{lemma}[subsection]{Lemma}

\newtheorem{conjecture}[subsection]{Conjecture}

\theoremstyle{definition}

\newtheorem{remark}[subsection]{Remark}

\numberwithin{equation}{subsection}

\def\calO{\mathcal{O}}

\def\gothU{\mathfrak{U}}

\def\AAA{\mathbb{A}}

\def\CC{\mathbb{C}}
\def\FF{\mathbb{F}}

\def\QQ{\mathbb{Q}}

\def\ZZ{\mathbb{Z}}

\def\1{\mathbf{1}}
\def\i{\mathbf{i}}
\def\j{\mathbf{j}}
\def\k{\mathbf{k}}

\def\rmB{\mathrm{B}}

\def\rmG{\mathrm{G}}

\def\rmM{\mathrm{M}}

\def\rmT{\mathrm{T}}

\newcommand{\MATRIX}[4]{{\big(\begin{smallmatrix}
			#1&#2\\#3&#4
		\end{smallmatrix}\big)}}

\DeclareMathOperator{\GL}{GL}
\DeclareMathOperator{\Hom}{Hom}
\DeclareMathOperator{\Ind}{Ind}

\DeclareMathOperator{\Sym}{Sym}

\newcommand{\Dig}{\mathrm{Dig}}
\newcommand{\Iw}{\mathrm{Iw}}

\newcommand{\Mult}{\mathrm{Mult}}
\newcommand{\rank}{\mathrm{rank}}
\newcommand{\Matrix}[4]{{\big(\begin{smallmatrix}
			#1 & #2 \\ #3 & #4
		\end{smallmatrix}\big)}}

\newcommand{\unr}{\mathrm{unr}}
\newcommand{\proj}{\mathrm{proj}}
\newcommand{\midd}{\mathrm{mid}}
\newcommand{\NP}{\mathrm{NP}}
\begin{document}
	\title{An explicit computation of the Hecke operator and the ghost conjecture}
	\author{Nha Xuan Truong}
	\address{Department of Mathematics, University of Hawaii at Manoa, Honolulu, Hawaii 96822
	}
	\email{nxtruong@hawaii.edu}
	\maketitle
	\setcounter{tocdepth}{1}
	\section{Introduction}
	Let $p$ be a prime number, and $N$ be a positive integer co-prime to $p$. For an integer $k\geq3$, we use $S_k\big(\Gamma_1(pN)\big)$  to denote the space of modular cuspforms of weight $k$, level $\Gamma_1(pN)$. A modular form $f$ in $S_k\big(\Gamma_1(pN)\big)$ can be written as
	$$f(\tau) = \sum_{n = 1}^{\infty} a_n(f)q^n, \  q = e^{2 \pi i \tau}. $$
	
	For each prime $l$, one can define a \emph{Hecke operator} $T_l$ on $S_k\big(\Gamma_1(pN)\big)$. When $l=p$, the Hecke operator is often denoted by $U_p$ instead. For an \emph{eigenform} $f$, after normalization, one can deduce that $$U_p f = a_p f.$$
	
	The $p$-adic valuation of $a_p$ is called the \emph{slope of the eigenform}. The study of slopes plays an important role in understanding the geometry of the so-called eigencurve introduced by Coleman and Mazur using  $p$-adic interpolation of overconvergent eigenforms \cite{Coleman-Mazur}. The eigencurve has many application in $p$-adic number theory, for example Kisin's proof of Fontaine-Mazur conjecture for $\GL_2$ \cite{kisin}. 
	
	The first numerical data of the slopes was due to Gouv\^ea and Mazur in \cite{gouvea0}  using computer calculations. Buzzard and his co-authors computed in the case with small primes $p$ and small level \cite{buzzard-calegari},\cite{buzzard-kilford},\cite{jacobs}. Liu, Wan and Xiao found the slopes for the $U_p$-operator over the boundary of the weight space to be certain unions of arithmetic progressions \cite{liu-wan-xiao}.
	Recently, Bergdall and Pollack proposed the ghost conjecture which predicts the slopes over the entire the weight space. In \cite{bergdall-pollack2}\cite{bergdall-pollack3}, the authors defined a formal power series, called the \emph{ghost series}, and conjectured that the slope of $U_p$ action is the same as the slopes of the Newton polygon of the ghost series. 
	
	We consider a variant of the ghost conjecture for overconvergent forms on definite quaternion algebras.The result can than be translate to modular forms using Jacquet-Langlands correspondence. The upshot is that, this circumvented the difficulties posed by the geometry of a modular curve, traded with the arithmetic complication of a quarternion algebra, which turns out to be more accessible by our method. 
	
	One simple and nice computation in this direction is due to Jacobs in \cite{jacobs}. In his thesis, Jacobs studied the case when $p=3$ with a particular level \cite{jacobs} and computed the slopes of the $U_3$-operation. The method is later refined by Wan--Xiao--Zhang \cite{wan-xiao-zhang}.
	
	In this paper, we investigate the Hecke operator $U_5$ and show that the $n \times n$ upper left minors of the matrix have non zero corank and, interestingly, follow that unimodal pattern in the ghost conjecture \cite{bergdall-pollack2}\cite{bergdall-pollack3}. This seems to give the ghost series of Bergdall and Pollack some theoretic explanation. We expect that the slopes of $U_5$-action in this case can be computed using an appropriate variant of the \emph{ghost series}, defined in (\ref{ghost series}). Assume this result, we achieve an upper bound for the slopes that is similar to the Gou\^ea's $\frac{k-1}{p+1}$ conjecture.

	The result can be generalized in my ongoing project with Ruochuan Liu, Liang Xiao and Bin Zhao to prove the ghost conjecture under a certain mild technical hypothesis. \\
	
	\textbf{Acknowledgments}. 
	I would like to thank my advisor, Liang Xiao, for introducing me to the topic, suggesting many ideas, reviewing early drafts and for constant encouraging and supporting.  Also, when working on this paper, the author had Graduate Fellowship funded by Xiao's NSF CAREER Grant DMS--1752703.
	
	I also would like to thank John Bergdall and Robert Pollack for their great idea \cite{bergdall-pollack} and thank all the people contributing to the SAGE software, as lots of my argument rely on first testing using a heavy computer simulation.
	\section{Setup}
	\subsection{The quaternion algebra}
Our setup is a variant of \cite{wan-xiao-zhang}. In this paper, we investigate the case $p=5.$

Let $D$ be a quaternion algebra over  $\mathbb{Q}$. Explicitly, we set
$$ D := \QQ\langle \i,\j\rangle/ (\i^2+1,\j^2+1, \i\j +\j\i).$$

We know that $D$ ramifies exactly at 2 and $\infty$, and splits at all odd primes $p$. In particular, $$D \otimes_\QQ \QQ_5 \cong \rmM_2(\QQ_5)$$ .

We denote $\nu_5$ as the square root of $-1$ in $\QQ_5$ that is congruent to $2$ modulo $5$.
In $5$-adic expansion, $$\nu_5 = 2+ 5+ 2\cdot 5^2 + \cdots.$$

We then fix an isomorphism between $D \otimes \QQ_5$ and $\rmM_2(\QQ_5)$ so that
\[\1 \leftrightarrow \begin{pmatrix}
	1 &0 \\ 0 & 1
\end{pmatrix}, \quad \i \leftrightarrow
\begin{pmatrix}
	0 & 1\\ -1& 0
\end{pmatrix}, \quad
\j \leftrightarrow \begin{pmatrix}
	\nu_5 & 0\\ 0 & -\nu_5
\end{pmatrix}.
\]

The result in the thesis is independent of the above isomorphism and the choice of the square root $\nu_5$. We denote $\k = \i\j$, and its image under the isomorphism is
$$\k \leftrightarrow \begin{pmatrix}
	0 & -\nu_5 \\ -\nu_5 & 0
\end{pmatrix}.$$
We also know the unit group of $\calO_D $ consists of $24$ elements:
\[
\calO_D^\times = \big\{ \pm\! \1, \pm \i, \pm \j, \pm \k, \tfrac12(\pm \1 \pm \i \pm \j \pm \k)\; \big\}.
\]
\newpage
	Their image under the above isomorphism are:
	\begin{center}
		\begin{tabular}{|c|c|c|}
			\hline $\calO_D^\times$&$\rmM_2(\QQ_5)$&mod 5\\
			
			\hline $\pm \1$  & $\pm \begin{pmatrix}1 & 0 \\ 0 & 1 \end{pmatrix}$ 
			& $\pm \begin{pmatrix}1 & 0  \\ 0 & 1 \end{pmatrix} $\\ 
			
			\hline $\pm \i $&$ \pm \begin{pmatrix}0 & 1 \\ -1 & 0\end{pmatrix}$
			& $ \pm \begin{pmatrix}	0 & 1  \\4 & 0 \end{pmatrix} $\\
			
			\hline $\pm \j  $& $\pm \begin{pmatrix}\nu_5 & 0 \\ 0 & -\nu_5 \end{pmatrix} $
			& $\pm \begin{pmatrix} 2 & 0  \\ 0 & 3 \end{pmatrix} $\\
			
			\hline $\pm \k  $&$ \pm \begin{pmatrix}	0 & -\nu_5 \\ -\nu_5 & 0 \end{pmatrix} $
			&$ \pm \begin{pmatrix} 0 & 3  \\ 3 & 0 \end{pmatrix} $\\
			\hline $\pm \frac{1}{2} (\1+\i+\j+\k) $&$ \pm \frac{1}{2} \begin{pmatrix}	1+\nu_5 & 1- \nu_5 \\ -(1+\nu_5) & 1-\nu_5 \end{pmatrix} $
			&$ \pm \begin{pmatrix}4 & 2  \\ 1 & 2 \end{pmatrix}$\\
			\hline $\pm \frac{1}{2} (-\1+\i+\j+\k) $
			&$ \pm \frac{1}{2} \begin{pmatrix}
				-1+\nu_5 & 1- \nu_5 \\ -(1+\nu_5) & -(1+\nu_5) \end{pmatrix} $
			&$ \pm \begin{pmatrix} 3 & 2  \\ 1 & 1 \end{pmatrix} $\\
			
			\hline $\pm \frac{1}{2} (\1-\i+\j+\k)$&$ \pm \frac{1}{2} \begin{pmatrix}
				1+\nu_5 & -(1+ \nu_5) \\ 1-\nu_5 & 1-\nu_5 \end{pmatrix}$
			&$\pm \begin{pmatrix}4 & 1  \\ 2 & 2 \end{pmatrix}$\\
			
			\hline $\pm \frac{1}{2} (\1+\i-\j+\k)$ &$ \pm \frac{1}{2}\begin{pmatrix}
				1-\nu_5 & 1- \nu_5 \\ -(1+\nu_5) & 1+\nu_5 \end{pmatrix}$
			&$\pm \begin{pmatrix}2 & 2  \\ 1 & 4 \end{pmatrix} $\\

			\hline $\pm \frac{1}{2} (\1+\i+\j-\k) $&$\pm \frac{1}{2} \begin{pmatrix}
				1+\nu_5 & 1+\nu_5 \\ -1+\nu_5 & 1-\nu_5 \end{pmatrix}$&$\pm \begin{pmatrix}
				4 & 4  \\ 3 & 2 \end{pmatrix} $\\
			
			\hline
			$\pm \frac{1}{2} (-\1-\i+\j+\k) $&$ \pm \frac{1}{2} \begin{pmatrix}
				-1+\nu_5 & -(1+ \nu_5) \\ 1-\nu_5 & -(1+\nu_5) \end{pmatrix}
			$&$ \pm \begin{pmatrix}
				3 & 1  \\ 2 & 1 \end{pmatrix} $\\
			
			\hline 
			$\pm \frac{1}{2} (-\1+\i-\j+\k) $&$ \pm \frac{1}{2} \begin{pmatrix}
				-(1+\nu_5) & 1- \nu_5 \\ -(1+\nu_5) & -1+\nu_5 \end{pmatrix}$&$ \pm \begin{pmatrix}
				1 & 2  \\ 1 & 3 \end{pmatrix} $\\
			
			\hline 
			$\pm \frac{1}{2} (-\1+\i+\j-\k) $&$ \pm \frac{1}{2} \begin{pmatrix}
				-1+\nu_5 & 1+ \nu_5 \\ -1+\nu_5 & -(1+\nu_5) \end{pmatrix}$&$ \pm \begin{pmatrix}
				3 & 4  \\ 3 & 1 \end{pmatrix} $\\
			
			\hline
			
		\end{tabular}
	\end{center}

	\subsection{Level structure} 
	We define the level structure using adeles. 
	The ring of finite adeles $\AAA_f$ is defined as 
	$$\AAA_f := \{ \left(x_\ell\right)_\ell \in \prod_{\ell \textrm{ prime}} \QQ_\ell|\textrm{ all but finitely many } x_\ell \in \ZZ_\ell  \}.$$
	
	We recall that $D$ splits at all odd prime  $\ell$, i.e $D \otimes_\QQ \QQ \cong\rmM_2(\QQ_\ell)$. We fix such an isomorphism for each $\ell$ and define 
	$$D_f = D\otimes_\QQ \AAA_f$$ 
	Then, 
	
	$$D_f^\times = (D \otimes_\QQ \AAA_f)^\times := \{ \left(x_\ell\right)_\ell \in (D \otimes \QQ_2)^\times  \prod_{\ell \textrm{ prime}, \ell \not = 2}\GL_2(\QQ_\ell)|\textrm{ all but finitely many } x_\ell \in \GL_2(\ZZ_\ell)\}.$$

	For $\ell = 2$, we use $D^\times(\ZZ_2)$ to denote the maximal compact subgroup of $  (D \otimes \QQ_2)^\times$.
	\newpage
	We consider the following open compact subgroup of $D^\times_f$:
	
	$$\widehat \Gamma_1(5)
	= D^\times (\ZZ_2) \times \prod_{\ell \neq 2, 5} \GL_2(\ZZ_\ell) \times \begin{pmatrix}
		\ZZ_5^\times & \ZZ_5\\ 5\ZZ_5&1+5\ZZ_5
	\end{pmatrix}$$
	and the maximal open compact subgroup
	$$\widehat \GL_2(\ZZ_5)
	= D^\times (\ZZ_2) \times \prod_{\ell \neq 2} \GL_2(\ZZ_\ell).$$
	
	We have the following useful lemma, which explains the choice of $\widehat \Gamma_1(5).$

	\begin{lemma}
		\label{gamma1}
		The following natural map is bijective.
		\begin{equation}
			\xymatrix@R=0pt{
				D^ \times \times  \widehat \Gamma_1(5) \ar[r] & D_f^\times\\
				(\delta, u) \ar@{|->}[r] & \delta u.}
		\end{equation}
	\end{lemma}
	\begin{proof}
		We follows the same argument in \cite{wan-xiao-zhang}. 
		By Lemma 1.22 in \cite{jacobs}, 
		$$D_f^\times = D ^\times \cdot  \widehat \GL_2(\ZZ_5).$$ 
		
		Since 	$D^\times \cap \widehat \GL_2(\ZZ_5) = \calO_D^\times$, taking into account of the duplication, we have
		\[
		D_f^\times = D ^\times \times^{\calO_D^\times}  \widehat \GL_2(\ZZ_5).
		\]
		
		So it suffices to check that the image of $\calO_D^\times$ in $\GL_2(\FF_5)$ forms a full set of coset representatives of $\widehat \GL_2(\ZZ_5)/\widehat \Gamma_1(5) \cong \GL_2(\FF_5) / \begin{pmatrix}
			\FF_5^\times & \FF_5\\ 0&1
		\end{pmatrix}.$\\ 
		This further follows from the fact that $\GL_2(\FF_5)$ and$ \begin{pmatrix}
			\FF_5^\times & \FF_5\\ 0&1
		\end{pmatrix} \times \calO_D^\times$  both have $24\cdot20$ elements and  $ \begin{pmatrix}
			\FF_5^\times & \FF_5\\ 0&1
		\end{pmatrix} \cap \calO_D^\times = \{1\}$.
	\end{proof}
	\subsection{Overconvergent automorphic forms}
	In this section, we define the space of overconvergent automorphic forms for a definite quaternion algebra and describe the Hecke actions explicitly.
	
	Fix an integer $k \in \ZZ$, consider the right action of the Iwahori subgroup $\Iw_{5,1} = \begin{pmatrix}
		\ZZ_5^\times & \ZZ_5\\ 5\ZZ_5& 1+5\ZZ_5
	\end{pmatrix}$ of $\GL_2(\ZZ_5)$ on $\QQ_5 \big< z \big>$ given by
	\begin{equation}
		\label{action}
		\textrm{for } \gamma = \big(
		\begin{smallmatrix}
			a&b\\c&d
		\end{smallmatrix} \big) \in \Iw_{5,1} \textrm{ and }f(z) \in \QQ_5 \big< z \big>, \quad
		(f||_k\gamma)(z): = (cz+d)^{k-2} f\big( \frac{az+b}{cz+d}\big).
	\end{equation}
	
	We define the space of \emph{overconvergent automorphic forms of weight $k$ and level $ \widehat \Gamma_1(5)$} to be
	\[
	S_k^{D,\dagger}(\widehat \Gamma_1(5)): = \Big
	\{
	\varphi: D^\times \backslash D_f^\times \to \QQ_5\big< z \big> \;\Big|\;
	\varphi(gu) = \varphi(g)||_k u_5, \textrm{ for any } g \in D_f^\times, u \in \widehat \Gamma_1(5)
	\Big\},
	\]
	where $u_5$ is the $5$-component of $u$.

	Using Lemma \ref{gamma1}, we have the isomorphism: 
	\begin{equation}
		\label{isomorphism}
		\xymatrix@R=0pt{
			S_k^{D,\dagger}(\widehat \Gamma_1(5)) \ar[r]^-\cong & \QQ_5 \big< z \big>\\
			\varphi \ar@{|->}[r] & \varphi(1).
		}
	\end{equation} 
	\begin{remark} The isomorphism gives a simpler description of the space of overconvergent automophic form, $S_k^{D,\dagger}(\widehat \Gamma_1(5))$, which allows us to compute explicitly the matrix of the Hecke action.   
	\end{remark}
\subsection{$U_5$-operator}
The space $S_k^{D,\dagger}(\widehat \Gamma_1(5))$ carries actions of Hecke operators $U_5$, defined as follows.

We write  
\begin{equation}
	\Iw_{5,1}\big(\begin{smallmatrix}
		5&0\\0&1
	\end{smallmatrix}\big) \Iw_{5,1} = \coprod_{i=0}^{4}
	\Iw_{5,1} v_i, \quad \textrm{with } v_i = \big(\begin{smallmatrix}
		5&0\\5i &1
	\end{smallmatrix}\big).
\end{equation}
Then the action of the operator $U_5$ on $S_k^{D,\dagger}\big(\widehat \Gamma_1(5)\big)$ is defined to be
\begin{equation}
	\label{E:defn of U_p}
	U_5(\varphi) = \sum_{i =0}^{4} \varphi|_k v_i, \quad \textrm{with }(\varphi|_k v_i)(g): = \varphi(gv_i^{-1})||_k v_i.
\end{equation}

	\section{Explicit computation of the infinite matrix of $U_5$}
	
	In terms of the explicit description of the space of overconvergent automorphic forms \eqref{isomorphism}, the $U_5$-operators can be described by the following commutative diagram.
	\[
	\xymatrix@C=80pt{
		S_k^{D,\dagger}(\widehat \Gamma_1(5)) \ar[r]^{\varphi \mapsto \varphi(1)} \ar[d]_{\begin{tiny}\varphi \mapsto U_5(\varphi) \end{tiny}} &\QQ_5 \big< z \big>
		\ar[d]_{\begin{tiny}\gothU_5\end{tiny}}
		\\
		S_k^{D,\dagger}(\widehat \Gamma_1(5)) \ar[r]^{\varphi \mapsto \varphi(1)} 
		&\QQ_5 \big< z \big>.}
	\]
	
	\begin{lemma}
		\label{Upmatrices}
		The $\gothU_5$  is given by
		$\gothU_5 = ||_\kappa \delta_1 + ||_\kappa \delta_2 + ||_\kappa \delta_3+||_\kappa \delta_4 + ||_\kappa \delta_5 $, where
		\[
		\delta_1 =
		\j-\mathbf{2}, \ \delta_2= \tfrac12(\1+\i-3\j-3\k), \ \delta_3= \tfrac12(\1+3\i-3\j+\k),\ \delta_4= \tfrac12(\1-3\i-3\j-\k),\ \textrm{and }\delta_5= \tfrac12(\1-\i-3\j+3\k).
		\]
		Their images in $\GL_2(\QQ_5)$ are given by
		\[
		\begin{pmatrix}
			\ -2+\nu_5 & 0 \\ 0 & -(2+ \nu_5)
		\end{pmatrix}, \quad
		\frac{1}{2}\begin{pmatrix}
			1-3 \nu_5 & 1+ 3\nu_5 \\ -1+ 3 \nu_5 & 1+3\nu_5
		\end{pmatrix}, \quad 
		\quad
		\frac{1}{2}\begin{pmatrix}
			1-3 \nu_5 & 3-\nu_5 \\ -(3+ \nu_5) & 1+3\nu_5
		\end{pmatrix}, \quad  
		\]
		\[    
		\frac{1}{2}\begin{pmatrix}
			1-3 \nu_5 & -3+\nu_5 \\  3 +\nu_5 & 1+3\nu_5
		\end{pmatrix}, \quad
		\textrm{and} \quad	
		\frac{1}{2}\begin{pmatrix}
			1-3 \nu_5 & -(1+ 3\nu_5) \\ 1- 3 \nu_5 & 1+3\nu_5
		\end{pmatrix}.
		\]
	\end{lemma}
	\begin{proof}
		We can compute $U_5$ explicitly
		\[
		U_5(\varphi)(1) = \sum_{j=1}^5 \varphi(v_j^{-1})||_\kappa v_j, \quad \textrm{for }v_j = \big(
		\begin{smallmatrix}
			5&0\\ 5 j &1
		\end{smallmatrix}
		\big)
		\]
		By Lemma \ref{gamma1}, we can write each $v_j^{-1}$ uniquely as $\delta_j^{-1} u_j$ for $\delta_j \in D^\times$ and $u_j \in \widehat \Gamma_1(5)$.
		Then
		\[
		\varphi(v_j^{-1})||_\kappa v_j =\varphi(\delta_j^{-1} u_j)||_\kappa v_j= \varphi(1)||_\kappa (u_{j,5}v_j) = \varphi(1)||_\kappa \delta_{j,5},
		\]
		where $u_{j,5}$ and $\delta_{j,5}$ denote the $5$-components of $u_j$ and $\delta_j$, respectively.
		
		We see that
		\[\delta_{j,5} = u_{j,5}v_{j,5} \in D^\times \cap \Iw_{5,1} v_j \subseteq D^\times\cap \Iw_{5,1}\big( \begin{smallmatrix}
			5 & 0 \\ 0 & 1
		\end{smallmatrix}\big) \Iw_{5,1} \subseteq D^\times \cap  U^* \big( \begin{smallmatrix}
			5 & 0 \\ 0 & 1
		\end{smallmatrix}\big),
		\]	
		where $U^*$ is a subgroup of $\GL_2(\ZZ_5)$ such that the lower right entry belongs to $1+ 5 \ZZ_5$	 
		If we put $\delta_j = \delta'_j (1 + 2 \j)$, then we have
		\begin{align}
			\nonumber
			\delta'_j &\in D^\times \cap  U^*\big(\begin{smallmatrix}
				5 & 0 \\ 0 & 1
			\end{smallmatrix}\big)(1+ 2\j)^{-1}  =  D^\times \cap U^* \big(\begin{smallmatrix}
				{1- 2 \nu_5} & 0 \\ 0 & \frac{(\1+2\nu_5)}{5}
			\end{smallmatrix}\big).
		\end{align}
		We know $D^\times \cap \widehat \GL_2(\ZZ_5) = \calO_D^\times$ and all $\delta'_j$ are distinct, so by taking modulo 5, and comparing to the list of $\calO_D^\times$, we have \\ $$\delta'_j \in \big \{ \j,-\frac{1}{2} (\1+\i+\j+\k),-\frac{1}{2} (\1-\i+\j+\k),-\frac{1}{2} (\1+\i+\j-\k),\frac{1}{2} (-\1+\i-\j+\k) \big \}.$$		
		It is then clear that all $\delta_j$'s are among the collections of the above right-multiplied by $\1+ 2\j$.  The rest of the lemma is straightforward.
	\end{proof}
	\section{ Computation}
	In this section we compute the matrix for the operator $U_5$ explicitly and explore its upper left $n \times n$ principal minor. 
	\begin{theorem}
		Let $(P_{i,j})_{i,j=0, 1, \cdots}$ denote the matrix for the operator $U_5$ on $\QQ_5 \langle z \rangle $, defined in \eqref{E:defn of U_p}, with respect to the power basis $1, z, z^2, \dots$, then
		\begin{itemize}
			\item 
			When $i=j$ \\$$ P_{ij}= \left(\frac{1+3\nu_5}{2}\right)^{k-i-2}\left(\frac{1-3\nu_5}{2}\right)^i\left(4\sum_{n=0}^i(-1)^{i-n}\binom{j}{n}\binom{k-j-2}{i-n} +(\nu_5-1)^{k-2}(\nu_5)^i \right);$$\\
			\item
			When  $i \neq j$ but $i \equiv j $mod 4 $$P_{i,j}=4\left(\frac{1+3\nu_5}{2}\right)^{k-i-2}\left(\frac{1-3\nu_5}{2}\right)^i \ \ \sum_{n=0}^{\min\{i,j\}}(-1)^{i-n}\binom{j}{n}\binom{k-j-2}{i-n};$$\\
			\item Otherwise $P_{ij} = 0.$
		\end{itemize} 
	\end{theorem}
	\begin{proof}
		
		The following argument is originally due to Jacobs \cite{jacobs}. 
		For a matrix $P=(P_{i,j})$, we define a generating series as the formal power series $H_P(x,y) := P_{i,j} x^iy^i.$
		
		We have an explicit expression of the generating series for the operator $U_p$ on $\QQ_5\langle z \rangle $ with respect to the power basis 
		
		$$
		H_P(x,y) : = \sum_{i,j \geq 0} P_{i,j} x^iy^i =\sum_{i=1}^5 \frac{(c_ix+d_i)^{k-1}}{c_ix+d_i-a_ixy-b_iy},
		$$
		where $\delta_i = \MATRIX {a_i}{b_i}{c_i}{d_i}$ are listed as in Lemma \ref{Upmatrices}.

		We include a proof here for the convenience of the readers. It's enough to compute the generating series of the operator $||_k\begin{pmatrix}
			a &b\\ c & d
		\end{pmatrix}$ acting on $\QQ_5\langle z \rangle $ with respect to the power basis. By definition,
		$$H(x,y)=\sum_{i\geq 0}y^i\left(cx+d\right)^{k-2}\left(\frac{ax+b}{cx+d}\right)^i=\frac{\left(cx+d\right)^{k-2}}{1-y\cdot\frac{ax+b}{cx+d}}=\frac{\left(cx+d\right)^{k-1}}{cx+d-axy-by}.$$
		$$H_{\delta_1}(x,y)=\frac{(-2-\nu_5)^{k-1}}{-2-\nu_5+(2-\nu_5)xy}= \frac{(-2-\nu_5)^{k-2}}{1-\frac{2-\nu_5}{2+\nu_5}xy} = (-2-\nu_5)^{k-2}\sum_{i\geq0} (\frac{2-\nu_5}{2+\nu_5})^i x^i y^i.$$
		
		Let $\alpha= \frac{1+3\nu_5}{2}$ then \\
		
		\begin{center}
			\begin{tabular}{|c|c|}
				\hline $\delta$&$H_\delta(x,y)$
				\\ \hline
				$\delta_2= \begin{pmatrix}
					\ 1-\alpha & \alpha \\ \alpha-1 & \alpha
				\end{pmatrix}$
				&$H_{\delta_2}(x,y)= \frac{((\alpha-1)x+\alpha)^{k-1}}{(\alpha-1)x+\alpha+(\alpha-1)xy-\alpha y}$
				\\ \hline
				$\delta_3= \begin{pmatrix}
					\ 1-\alpha & -\alpha\nu_5 \\ (\alpha-1)\nu_5 & \alpha
				\end{pmatrix}$&$H_{\delta_3}(x,y)= \frac{((\alpha-1)\nu_5x+\alpha)^{k-1}}{(\alpha-1)\nu_5x+\alpha+(\alpha-1)xy+\alpha \nu_5 y}$
				\\ \hline
				$\delta_4= \begin{pmatrix}
					\ 1-\alpha & \alpha\nu_5 \\ (1-\alpha)\nu_5 & \alpha
				\end{pmatrix}$
				&$H_{\delta_4}(x,y)= \frac{((1-\alpha)\nu_5x+\alpha)^{k-1}}{(1-\alpha)\nu_5x+\alpha+(\alpha-1)xy-\alpha \nu_5 y}$
				\\ \hline
				$\delta_5= \begin{pmatrix}
					\ 1-\alpha & -\alpha \\ 1-\alpha & \alpha
				\end{pmatrix}$
				&$H_{\delta_5}(x,y)= \frac{((1-\alpha)x+\alpha)^{k-1}}{(1-\alpha)x+\alpha+(\alpha-1)xy-\alpha y}$
				\\ \hline
			\end{tabular}
		\end{center}
		
		We observe that $$H_{\delta_3}(x,y)=H_{\delta_2}(\nu_5x,-\nu_5y), H_{\delta_4}(x,y)=H_{\delta_2}(-\nu_5x,\nu_5y), \textrm{and } H_{\delta_5}(x,y)=H_{\delta_2}(-x,y).$$Hence, in the formal power series of the sum $\sum_{i=2}^{5}H_{\delta_i}(x,y)$, the coefficient of $x^iy^j$ is nonzero if and only if  $i\equiv j$  (mod 4), and in this case it is equal to 4 times the coefficient of $x^iy^j$ in $H_{\delta_2}(x,y)$. We are left to compute the power series $\sum a_{i,j}x^iy^j$ of $H_{\delta_2}(x,y)$. Rewriting $H_{\delta_2}(x,y)$ in the formal power series form, we have
		
		$$a_{i,j} = (\frac{1+3\nu_5}{2})^{k-i-2}(\frac{1-3\nu_5}{2})^i\sum_{n=0}^{\min\{i,j\}}(-1)^{i-n}\binom{j}{n}\binom{k-j-2}{i-n}.$$ 
		This combining with our computation of $H_{\delta_1}$ proves the lemma. 
	\end{proof}
	
	\begin{remark}We note that $P_{i,j} =0$ unless $i \equiv j $(mod4). Hence, we would like to decompose the space $S_k^{D,\dagger}(\widehat \Gamma_1(5))$ into 4 sub-spaces corresponding to the four sub-matrices $\big( P_{4i+a,4j+a}\big)_{i,j = 0, 1, \dots} (a=0,..,3)$. The decomposition comes from the study of the Hecke operator at $p=2$. \\
	\end{remark}
\subsection{$U_2$-operator}
		The quaternion algebra $D$ is ramified at $p=2$, i.e. $$D \otimes_\QQ \QQ_2 \cong \QQ_4[\pi_2]/\left( \pi_2^2=2, \pi_2a = 
		\mathrm{Fr}_2(a) \pi_2 \right)$$ where $\mathrm{Fr}_2$ is the Frobenius endomorphism of $\QQ_4$. We define the Hecke operator $U_2$ to be the map
		
		\begin{equation}
			\label{U_2_action}
			\xymatrix@R=0pt{
				U_2: S_k^{D,\dagger}(\widehat \Gamma_1(5)) \ar[r]& S_k^{D,\dagger}(\widehat \Gamma_1(5)\\
				\varphi(g) \ar@{|->}[r] & \varphi(g\pi_2).
			}
		\end{equation}
		We can write $U_2$ explicitly following the same recipe for $U_5$-operator. 
		
		\[
		\xymatrix@C=80pt{
			S_k^{D,\dagger}(\widehat \Gamma_1(5)) \ar[r]^{\varphi \mapsto \varphi(1)} \ar[d]_{\begin{tiny}\varphi \mapsto U_2(\varphi) \end{tiny}} &\QQ_5 \big< z \big>
			\ar[d]_{\begin{tiny}\gothU_2\end{tiny}}
			\\
			S_k^{D,\dagger}(\widehat \Gamma_1(5)) \ar[r]^{\varphi \mapsto \varphi(1)} 
			&\QQ_5 \big< z \big>.}
		\]
		\begin{lemma}
			\label{U_2}
			$\mathfrak{U}_2 \big(\varphi(z)\big) = \left( \frac{-1-\nu_5}{2}\right)^{k-2} \varphi(-\nu_5z).$
		\end{lemma}
		\begin{proof}
			We have	
			\begin{align*}
				\varphi(\pi_2)=&\varphi\left(((1+\j)^{-1},(1+\j)^{-1},(1+\j)^{-1},\dots) (\pi_2,1,1,\dots)\right)\\=& \varphi(\frac{\pi_2}{1+\j},\frac{1}{1+\j},\frac{1}{1+\j},\dots)=\varphi(1,1,(1+\j)^{-1},\dots)
			\end{align*}
			Here, the last equality holds because at all places $\ell \neq 5$,$(1+\j)^{-1}$  is belong to $\GL_2(\ZZ_l)$). 
			
			Hence, 
			\begin{align*}
				U_2 \varphi(1)(z)=&\varphi((1+j)^{-1})(z)= \varphi(-1)|_{-(1+j)^{-1}}(z)\\=& \varphi|_{\begin{pmatrix}
						\frac{-1+\nu_5}{2} & 0 \\ 0 & \frac{-1-\nu_5}{2}
				\end{pmatrix}}(z) = \left( \frac{-1-\nu_5}{2}\right)^{k-2} \varphi(1)(-\nu_5z).
			\end{align*}
		\end{proof}
		
		The lemma gives us the decomposition: 
		$$S_k^{D,\dagger}(\widehat \Gamma_1(5))=\bigoplus_{a=0,\dots,3} S_k^{D,\dagger}(\widehat \Gamma_1(5))^{U_2= \left( \frac{1+\nu_5}{2}\right)^{k-2}\cdot \left(-\nu_5\right)^a}$$ 
		
		To ease the notation, we write $S_{k,a}^{D,\dagger}(\widehat \Gamma_1(5)) :=S_k^{D,\dagger}(\widehat \Gamma_1(5))^{U_2= \left( \frac{1+\nu_5}{2}\right)^{k-2} \cdot \left(-\nu_5\right)^a}$\\
			\section{SAGE computation of the $n \times n$ minor}
		Explicitly, $S_{k,a}^{D,\dagger}(\widehat \Gamma_1(5))$ corresponds under \eqref{isomorphism} to $z^a\cdot \QQ_5\big< z^4 \big>.$  
		In each of the four sub-spaces, we observe a surprising fact about the upper left minor matrices. \\
		Let $ P_n(k,a) =(P_{4i+a,4j+a})_{0\leq i,j\leq n-1}$ be the upper $n \times n$ left minor  and choose the weight $k=4k_\bullet+2a-2$. 
		
		We use SAGE to compute the co-rank of these minors. 
		The result shows that for lots of $k_\bullet$, $P_n(k,a)$ does not always have full rank as show in the following tables. The blank places mean that the corresponding corank is zero. 
		
	\begin{center}
		{Corank of $P_n(4k_\bullet-2,0)$}
		\begin{tabular}{|c|c|c|c|c|c|c|c|c|c|c|c|c|c|c|c|c|c|c|}
			\hline $n \times n$ &
			$k_\bullet$ =  4 &  5 & 6 & 7 & 8 & 9 & 10 & 11 & 12&13&14&15&16&17&18&19&20&21 
			\\ \hline
			$2 \times 2$&1&1&1&&1&&&&&&&&&&&&&
			\\ \hline
			$3 \times 3$&&1&2&1&2&1&1&1&1&&1&&&&&&&
			\\ \hline
			$4 \times 4$&&&1&1&3&2&2&2&2&1&2&1&1&1&1&&1&
			\\ \hline
			$5 \times 5$ &&&&&2&2&3&3&3&2&3&2&2&2&2&1&2&1
			\\ \hline
			$6 \times 6$& & & & &1& 1& 2& 3& 4& 3& 4& 3&3&3&3&2&3&2
			\\ \hline
			$7 \times 7$ &&&&&&&1&2&3&3&5&4&4&4&4&3&4&3
			\\ \hline
			$8 \times 8$ &&&&&&&&1&2&2&4&4&5&5&5&4&5&4
			\\ \hline
			$9 \times 9$ &&&&&&&&&1&1&3&3&4&5&6&5&6&5
			\\ \hline
			$10 \times 10$ &&&&&&&&&&&2&2&3&4&5&5&7&6
			\\ \hline
		\end{tabular}
	\end{center}
	\medskip
	\begin{center}
		
		{Corank of $P_n(4k_\bullet,1)$}
		\begin{tabular}{|c|c|c|c|c|c|c|c|c|c|c|c|c|c|c|c|c|c|c|}
			\hline $n\times n$&
			$k_\bullet$ = 3 & 4 &  5 & 6 & 7 & 8 & 9 & 10 & 11 & 12&13&14&15&16&17&18&19&20 
			\\ \hline
			$2 \times 2$&1&1&1&1&1&&1&&&&&&&&&&&
			\\ \hline
			$3 \times 3$&&&1&2&2&1&2&1&1&1&1&&1&&&&&
			\\ \hline
			$4 \times 4$&&&&1&2&2&3&2&2&2&2&1&2&1&1&1&1&
			\\ \hline
			$5 \times 5$ &&&&&1&1&3&3&3&3&3&2&3&2&2&2&2&1
			\\ \hline
			$6 \times 6$& & & & && & 2& 2& 3& 4& 4& 3&4&3&3&3&3&2
			\\ \hline
			$7 \times 7$ &&&&&&&1&1&2&3&4&4&5&4&4&4&4&3
			\\ \hline
			$8 \times 8$&&&&&&&&&1&2&3&3&5&5&5&5&5&4
			\\ \hline
			$9 \times 9$ &&&&&&&&&&1&2&2&4&4&5&6&6&5
			\\ \hline
			$10 \times 10$ &&&&&&&&&&&1&1&3&3&4&5&6&6		
			\\ \hline		
		\end{tabular}
	\end{center}
	\medskip   
	\begin{center}
		
		{Corank of $P_n(4k_\bullet+2,2)$}
		
		\begin{tabular}{|c|c|c|c|c|c|c|c|c|c|c|c|c|c|c|c|c|c|c|}
			\hline $n\times n$&
			$k_\bullet$ = 3 & 4 &  5 & 6 & 7 & 8 & 9 & 10 & 11 & 12&13&14&15&16&17&18&19&20 
			\\ \hline
			$2 \times 2$&&2&1&1&1&1&&1&&&&&&&&&&
			\\ \hline
			$3 \times 3$&&1&1&2&2&2&1&2&1&1&1&1&&1&&&&
			\\ \hline
			$4 \times 4$&&&&1&2&3&2&3&2&2&2&2&1&2&1&1&1&1
			\\ \hline
			$5 \times 5$ &&&&&1&2&2&4&3&3&3&3&2&3&2&2&2&2
			\\ \hline
			$6 \times 6$& & & & && 1& 1& 3& 3& 4& 4& 4&3&4&3&3&3&3
			\\ \hline
			$7 \times 7$ &&&&&&&&2&2&3&4&5&4&5&4&4&4&4
			\\ \hline
			$8 \times 8$&&&&&&&&1&1&2&3&4&4&6&5&5&5&5
			\\ \hline
			$9 \times 9$ &&&&&&&&&&1&2&3&3&5&5&6&6&6
			\\ \hline
			$10 \times 10$ &&&&&&&&&&&1&2&2&4&4&5&6&7		
			\\ \hline
		\end{tabular}
	\end{center}
	\medskip   
	
	\begin{center}
		
		{Corank of $P_n(4k_\bullet+4,3)$}
		\begin{tabular}{|c|c|c|c|c|c|c|c|c|c|c|c|c|c|c|c|c|c|c|}
			\hline $n\times n$&
			$k_\bullet$ = 2 & 3 & 4 &  5 & 6 & 7 & 8 & 9 & 10 & 11 & 12&13&14&15&16&17&18&19
			\\ \hline
			$2 \times 2$&&1&1&2&1&1&1&1&&1&&&&&&&&
			\\ \hline
			$3 \times 3$&&&&2&2&2&2&2&1&2&1&1&1&1&&1&&
			\\ \hline
			$4 \times 4$&&&&1&1&2&3&3&2&3&2&2&2&2&1&2&1&1
			\\ \hline
			$5 \times 5$&&&&&&1&2&3&3&4&3&3&3&3&2&3&2&2
			\\ \hline
			$6 \times 6$&&&&&&&1&2&2&4&4&4&4&4&3&4&3&3
			\\ \hline
			$7 \times 7$ &&&&&&&&1&1&3&3&4&5&5&4&5&4&4
			\\ \hline
			$8 \times 8$&&&&&&&&&&2&2&3&4&5&5&6&5&5
			\\ \hline
			$9 \times 9$ &&&&&&&&&&1&1&2&3&4&4&6&6&6
			\\ \hline
			$10 \times 10$ &&&&&&&&&&&&1&2&3&3&5&5&6		
			\\ \hline
		\end{tabular}
	\end{center}
	\medskip
		The corank of $P_n(k,a)$ shows the unimodal pattern that looks like the multiplicity of the ghost zeros \cite{bergdall-pollack}. For each fixed weight $k$, as $n$ increase, the corank $P_n(k,a)$ pattern is $$1,2,3, \dots, 3,2,1$$
		In the next chapter, we will state this pattern in a more precise form in term of the dimensions of the spaces of classical automorphic forms.

	\section{The classical automorphic forms}
	From now on, we consider only the weight k such that $k= 4k_\bullet + 2a -2$.
	
	We recall spaces of classical automorphic forms and compute their dimensions.
	\begin{eqnarray}
		\nonumber
		& S_{k}^{D}(\widehat \GL_2(\ZZ_5)) = \big\{ \varphi: D^\times \backslash D_f^\times \to \QQ_5[z]^{\deg \leq k-2} \ \textrm{s.t. } \varphi(x u) = \varphi(x)||_k{u_5} \textrm{ for all }u \in \widehat \GL_2(\ZZ_5) \big\}.
		\\
		\nonumber
		&S_{k}^{D}(\widehat \Gamma_1(5)) = \big\{ \varphi: D^\times \backslash D_f^\times \to \QQ_5[z]^{\deg \leq k-2} \ \textrm{s.t. } \varphi(x u) = \varphi(x)||_{k}{u_5} \textrm{ for all }u \in \widehat \Gamma_1(5) \big\}.
	\end{eqnarray}
	
	Here we extend the action of $\widehat \Gamma_1(5)$ defined in (\ref{action}) to $\widehat \GL_2(\ZZ_5).$
	The action $U_2$, defined in  \eqref{U_2_action} stabilizes each of $S_{k}^{D}(\widehat \GL_2(\ZZ_5))$ and $S_{k}^{D}(\widehat \Gamma_1(5))$, and decomposes these space into 4 sub-spaces  $S_{k,a}^{D}(\widehat \GL_2(\ZZ_5)), S_{k,a}^{D}(\widehat \Gamma_1(5))$ with the dimension $d_{k,a}^\unr$ and $d_{k,a}^\Iw$ respectively. We can compute them explicitly. 
	
\begin{theorem}\label{dim_iw} For integers $k= 4k_\bullet + 2a -2$, we have 
	$d_{k,a}^\Iw = k_\bullet$.
\end{theorem}
\begin{proof}
	The isomorphism in (\ref{isomorphism}) induces the following commutative diagram
	\[
	\xymatrix@C=80pt{
		S_k^D(\widehat \Gamma_1(5))  \ar@{^{(}->}[r] \ar[d]_{\cong} &S_k^{D,\dagger}(\widehat \Gamma_1(5))
		\ar[d]_{\cong}
		\\
		\QQ_5[z]^{\deg \leq k-2} \ar@{^{(}->}[r]
		&\QQ_5 \big< z \big>.}
	\]
	Since the $U_2$-action is compatible with vertical maps being isomorphism, we have
	$$S_{k,a}^{D}(\widehat \Gamma_1(5)) \cong \QQ_5 z^a \oplus \QQ_5 z^{4+a} \oplus \cdots = \bigoplus_{0 \leq i\leq k-2,\\ i\equiv a \textrm{ mod 4}} \QQ_5z^{i}$$
	The dimension formula follows easily.
\end{proof}
It is more difficult to compute $d_{k,a}^\unr$, which requires some representation theory.  
We use the following well-known fact from the cohomology theory of finite groups: 
\begin{lemma}
	If $G$ is a finite group and M is any G-module, then $H^i(G,M)$ is annihilated by the order of the group $G$ for any $i \geq 1$.
\end{lemma}
As a corollary,  if $p \nmid \|G\|$, then $-^G$ is exact on $\ZZ_p[G]-$modules.
In our case, the order of $\calO_D^\times$ is 24 which is coprime to 5, thus $H^i(\calO_D^\times,M) = 0$ for all $\ZZ_5 \left[\calO_D^\times\right] $-module $M$ and $i \geq 1$. Hence, we have $H^0\left(\calO_D^\times,M\right) = M^{\calO_D^\times}$ is exact in $M$. 

Let us consider the right representation $\Sym^{a}\QQ_5^{\oplus 2} \otimes \det^b$ of $\GL_2(\QQ_5)$. We may identify $\Sym^{a}\QQ_5^{\oplus 2} \otimes \det^b$ with the vector space $\QQ_5\left[z\right]^{\deg \leq a}$, where the action of $\GL_2(\QQ_5)$ is given by:

\begin{equation}
	f||_{\big(\begin{smallmatrix}
			\alpha&\beta\\\gamma&\delta
		\end{smallmatrix}\big)}(z): = (\alpha \delta - \beta \gamma)^b(\gamma z+\delta)^a f\big( \frac{\alpha z+\beta}{\gamma z+\delta}\big).
\end{equation}

When $a=k-2$ and $b =0$, we get the $k-2$th symmetric power $\Sym^{k-2} \QQ_5^{\oplus 2}$.
The following isomorphism is equivariant   for the $ \GL_2(\ZZ_5)$-action. 
\begin{equation}
	\xymatrix@R=0pt{\Sym^{k-2}\QQ_5^{\oplus 2} \ar@{<->}[r] & \QQ_5[z]^{\leq k-2} \\
		\sum_{i =0}^{k-2} a_i X^i Y^{k-2-i} \ar@{<->}[r] & \sum_{i=0}^{k-2} a_i z^i.
	}
\end{equation}

In the proof of Lemma \ref{gamma1}, we recall that 	\[
D_f^\times = D ^\times \times^{\calO_D^\times}  \widehat \GL_2(\ZZ_5).
\] Hence,  

$$ S_{k}^{D}(\widehat \GL_2(\ZZ_5)) \cong \Hom_{\GL_2[\ZZ_5]}(D^\times \backslash D_f^\times / \widehat \GL_2(\ZZ_5),\, \Sym^{k-2} \QQ_5^{\oplus 2}) \cong \left(\Sym^{k-2} \QQ_5^{\oplus 2}\right)^{\calO_D^\times}.$$

Moreover, we see that  $$  
\left( \Sym^{k-2} \QQ_5^{\oplus 2}\right)^{\calO_D^\times} \cong  \left( \Sym^{k-2} \ZZ_5^{\oplus 2}\right)^{\calO_D^\times}\left[\frac{1}{5}\right]. $$
Since the order of $\calO_D^\times$ is not divisible by 5,

$$  \left( \Sym^{k-2} \ZZ_5^{\oplus 2}\right)^{\calO_D^\times}/5 \cong \big(\Sym^{k-2} \FF_5^{\oplus 2}\big)^{\calO_D^\times}.$$

Thus, $\dim_{\QQ_5}\left( \Sym^{k-2} \QQ_5^{\oplus 2}\right)^{\calO_D^\times} = 
\rank_{\ZZ_5}\left( \Sym^{k-2} \ZZ_5^{\oplus 2}\right)^{\calO_D^\times} = \rank_{\FF_5}\left( \Sym^{k-2} \FF_5^{\oplus 2}\right)^{\calO_D^\times}.$\\

We would like to find the Jordan--H\"older factors of $\Sym^{k-2} \FF_5^{\oplus 2}$ as a represenation of $\FF_5\left[\GL_2\left(\FF_5\right)\right]$.

By Proposition 2.17 in \cite{breuil}, the Serre weights $\left(\Sym^{m} \FF_5^{\oplus 2} \otimes \det^n\right)$ for $m = 0,1,\dots,4$ and $n=0,\dots,3$ exhaust all irreducible representations of $\GL_2( \FF_5)$ over $\FF_5.$ 
We use $\sigma_{m,n}$ to denote the Serre weight  $\left(\Sym^{m} \FF_5^{\oplus 2} \otimes \det^n\right)$.
By checking the action of all element of $\calO_D^\times$, we have the following lemma.

\begin{lemma}
	$ \left(\sigma_{m,n}\right)^{\calO_D^\times}$ is non-zero if and only if $m=0$.
\end{lemma} 
We recall that $U_2$-action on $S_{k}^{D}(\widehat \GL_2(\ZZ_5)) \cong  \left(\Sym^{m} \QQ_5^{\oplus 2} \otimes \det^n\right)^{\calO_D^\times}$ translates to the action of $\begin{pmatrix}
	\frac{-1+\nu_5}{2} & 0 \\ 0 & \frac{-1-\nu_5}{2}
\end{pmatrix}$ as in Lemma \ref{U_2}. In the case when $ \left(\sigma_{m,n}\right)^{\calO_D^\times}$ is nonzero, i.e $m=0$, $\sigma_{0,n}$ is one dimension on which $U_2$ acts by  $\begin{pmatrix}
	3 & 0 \\ 0 & 1
\end{pmatrix}$. 
Hence, each $\sigma_{0,a}$ in the  multiset of Jordan-H\"older factors of $\sigma_{k-2,0}$ contributes one to the dimension $d_{k,a}^\unr$. Let $\Mult_{\sigma_{m,n}}\left(\sigma_{k-2,0}\right)$ be the multiplicity of $\sigma_{m,n}$ in the multiset of Jordan--H\"older factor of $\sigma_{k-2,0}$. We have:
\begin{equation}
	\label{dimMult}
	d_{k,a}^\unr= \Mult_{\sigma_{0,a}}\left(\sigma_{k-2,0}\right)
\end{equation}

From now on, we use the notation \[\delta_{a,b,c}= \begin{cases}
	&1 \textrm{ if } a \equiv b \bmod c.
	\\&
	0 \textrm{ otherwise.}	
\end{cases}\]
\begin{theorem}\label{dim_unr}  For integers $k= 4k_\bullet + 2a -2$, we have 
	$d_{k,a}^\unr = \lfloor \frac{k_\bullet-a+5}{6}\rfloor -\delta_{k_\bullet,a+2,6}$.
\end{theorem}
\begin{proof}
	It is enough to prove the equality in the Gronthediek group 
	
	Let $\rmB = \begin{pmatrix}{\FF_5^\times}&{\FF_5}\\{0}&{\FF_5^\times} \end{pmatrix}$ and $\rmG = \GL_2(\FF_5)$
	We consider the following character $\eta$ of $ \rmB$
	$$
	\eta: \rmB\rightarrow \FF_5^\times,\qquad  \Matrix {\alpha}{\beta}0{\delta} \mapsto \alpha.
	$$
	For $k \in \ZZ$, we have an induced representation
	$$ \Ind_{\rmB}^{\rmG}(\eta^k): = \{f: \rmG \to \FF\; |\; f(bg) = \eta^k( b)f( g), \textrm{ for all } b \in  \rmB\},$$
	which we equip with the \emph{right action} given by
	$$ f|_{h}(g): = f( g h^\rmT), \quad \textrm{for }  g, h \in  \rmG.$$
	This is the transpose of the usual left action.
	We use the following two lemmas from \cite[Lemma~4.9]{paskunas}, \cite[Proposition~2.4]{Reduzzi}, respectively. 
	\begin{lemma}
		\label{L:exact sequence of induced representation}
		For any integer $1\leq k\leq p-1$, we have an exact sequence of
		$\rmG$-representations:
		\begin{equation}
			\nonumber
			0\rightarrow \sigma_{k,0}\rightarrow
			\Ind_{\rmB}^{\rmG}\eta^{k}\rightarrow \sigma_{p-1-k, k}\rightarrow 0.\end{equation}
	\end{lemma}
	\begin{lemma}
		\label{L:exact sequence of (p+1)-steps jump}
		For any integer $k\geq p+1$, we have an exact sequence of
		$\rmG$-representations:
		$$
		0\rightarrow \sigma_{k-(p+1),1}\xrightarrow{i}
		\sigma_{k,0}\xrightarrow{\pi} \Ind_{ \rmB}^{ \rmG}\eta^{k}\rightarrow
		0.
		$$
	\end{lemma}
	
	Let $p = 5$, for $k \in \ZZ$, denote $\bar k = k \mod 4$, from the two exact sequences above,  we have the equality in $\mathrm{Groth}(\FF_5[\GL_2(\FF_5)])$
	\begin{equation}
		\label{Grothequa}
		[\sigma_{k-2,0}]-[\sigma_{k-8, 1}]=[\sigma_{\overline {k-2},0}]+[\sigma_{4-\overline {k-2}, \overline {k-2}}].
	\end{equation}
	We consider the case $k \mod 4 =2$. The other case is very similar. The equality \eqref{Grothequa} becomes: 
	$$[\sigma_{k-2,0}]-[\sigma_{k-8, 1}]=[\sigma_{0,0}]+[\sigma_{4,0}].$$
	
	Replacing $k-2$ by $k-8$ and then twisting by character det, we have
	
	$$[\sigma_{k-8,1}]-[\sigma_{k-14, 2}]=[\sigma_{2,1}]+[\sigma_{2,3}].$$

	Continuing this way, 
	$$[\sigma_{k-14,2}]-[\sigma_{k-20, 3}]=[\sigma_{0,2}]+[\sigma_{4,2}].$$
	$$[\sigma_{k-20,1}]-[\sigma_{k-26, 0}]=[\sigma_{2,3}]+[\sigma_{2,1}].$$
	
	Adding up these equalities, on left side we have $[\sigma_{k-2,0}]-[\sigma_{k-26, 0}]$ and on the right side the Serre weight $\sigma_{0,a}$ for $a$ = 0 or 2 appears exactly once. Combining this with \eqref{dimMult}, we showed that  $d_{k+24,a}^\unr = d_{k,a}^\unr + 1$. For $k \leq 26$, by counting the Serre weight $\sigma_{0,a}$, we can compute the value $d_{k,a}^\unr$ as in the table. \\
	\begin{center}
		\begin{tabular}{|c|c|c|c|c|}
			\hline $k_\bullet$ &
			$d_{4k_\bullet-2,0}^\unr$& $d_{4k_\bullet,1}^\unr$ & $d_{4k_\bullet+2,2}^\unr$ & $d_{4k_\bullet+4,3}^\unr$ 
			\\ \hline
			$1$&$1$&$0$&$0$&$0$
			\\ \hline
			$2$&$0$&$1$&$0$&$0$
			\\ \hline
			$3$&$1$&$0$&$1$&$0$
			\\ \hline
			$4$&$1$&$1$&$0$&$1$
			\\ \hline
			$5$&$1$&$1$&$1$&$0$
			\\ \hline
			$6$&$1$&$1$&$1$&$1$
			\\ \hline
		\end{tabular}
	\end{center}	
	From this, we conclude that 
	$$d_{k,a}^\unr = \lfloor \frac{k_\bullet-a+5}{6}\rfloor -\delta_{k_\bullet,a+2,6}.$$
	
\end{proof}
	\section{Main Theorem}
\begin{theorem}
	For an integer $k \equiv 2a-2 \bmod 4$ and an integer $n$ such that $d_{k,a}^\unr \leq n \leq d_{k,a}^\Iw -d_{k,a}^\unr$, the corank of $P_n(k,a)$ is at least 
	\begin{equation}
		\label{E:rank of n*n minor}
		\min\{ n - d_{k,a}^\unr, d_{k,a}^\Iw - d_{k,a}^\unr -n\}.
	\end{equation}
\end{theorem}
\begin{proof}
	Consider the maps $i:S_{k,a}^D(\widehat \GL_2(\ZZ_5)) \to S_{k,a}^D(\widehat \Gamma_1(5))$ and $proj: S_{k,a}^D(\widehat \Gamma_1(5)) \to S_{k,a}^D(\widehat \GL_2(\ZZ_5))$, given by
	\[  
	\quad i(\varphi)(x) = \varphi(x\MATRIX 5001^{-1})|_{\MATRIX 5001}, \quad \textrm{and} \quad \proj(\varphi)(x) = \varphi(x \MATRIX 0110)|_{\MATRIX 0110}+ \sum_{j=0}^{4} \varphi(x\MATRIX 10j1^{-1})|_{\MATRIX 10j1}
	\]
	
	We note that the map $i$ is the analogue of the map $f(z)\ \mapsto f(pz)$ for modular form. For the formula for proj, the matrices are chosen from a coset of $$\widehat \Gamma_0(5) \backslash \widehat \GL_2(\ZZ_5) = \{\MATRIX 0110 \MATRIX100d, \MATRIX 10j1 \MATRIX100d, \textrm{ for } j=0,\dots,4, d =1,\dots,4\}.$$As $S_{k,a}^D(\widehat \Gamma_1(5))$ is invariant under the action of $\MATRIX100d$, the map proj is well defined.\\ 
	
	We have the key identity on the space $S^D_{k,a}(\widehat\Gamma_1(5))$
	\[
	i \circ proj(\varphi)(x)= \sum_{j=0}^{4} \varphi(x\MATRIX 10j1^{-1}\MATRIX 5001^{-1})|_{\MATRIX 10j1\MATRIX 5001} + \varphi(x \MATRIX 0110)\MATRIX 5001^{-1}|_{\MATRIX 0110\MATRIX 5001} \]\[=\sum_{j=0}^{4} \varphi(x\MATRIX 50{5j}1^{-1})|_{\MATRIX 50{5j}1} + \varphi(x \MATRIX 0150^{-1})|_{\MATRIX 0150} = \mathfrak{U}_5 \big(\varphi(z)\big) + \varphi(x \MATRIX 0150^{-1})|_{\MATRIX 0150}
	\]
	If we write this equality as matrices, it means that 
	\begin{equation}
		\label{U_5 matrix}
		(P_{4i+a,4j+b})_{0 \leq 4i+a,4j+a \leq k-2} = \textrm{matrix for }i \circ \proj \ - \ \textrm{matrix for }\varphi \mapsto \varphi(\MATRIX 0150^{-1})|_{\MATRIX 0150}.
	\end{equation}  
	The matrix for $i \circ \proj$ has rank not greater than $d_{k,a}^\unr$. To see the matrix of the second operation we use the same argument in Lemma \ref{Upmatrices},we write $\MATRIX 0150^{-1} = \MATRIX{0}{-2-\nu_5}{-2+\nu_5}{0}^{-1} \MATRIX{-2-\nu_5}{0}{0}{\frac{-2+\nu_5}5}$. Then 
	$$
	\varphi(x\MATRIX 0150^{-1})|_{\MATRIX 0150} = \varphi(x)|_{\MATRIX{0}{-2-\nu_5}{-2+\nu_5}{0}}
	$$
	In particular, this means that the matrix of the second operator of \eqref{U_5 matrix} on $S_{k}^D(\widehat \Gamma_1(5))$ is the anti-diagonal matrix
	$$
	\begin{small}\begin{pmatrix}
			0 & 0 & \cdots &0 & (-2+\nu_5)^a(-2-\nu_5)^{k-2-a}\\
			0 & 0 & \cdots & (-2+\nu_5)^{a+4}(-2-\nu_5))^{k-6-a}  &0\\
			\vdots & \vdots & & \vdots & \vdots
			\\
			(-2+\nu_5)^{k-2-a}(-2-\nu_5)^a & 0 & \cdots & 0&0
		\end{pmatrix}
	\end{small}.
	$$
	But if we only look at its upper left $n\times n$ submatrix, its rank is
	\begin{equation}
		\label{E:rank of AL}
		\begin{cases}
			0 & \textrm{if } n \leq d_{k,a}^\Iw/2
			\\
			2n-d_{k,a}^\Iw & \textrm{ otherwise}.
	\end{cases}\end{equation}
	Thus, the rank of $P_n(k,a)$ is at most the sum of \eqref{E:rank of AL} and $d_{k,a}^\unr$. We deduce that the corank of $P_n(k,a)$ is at least \eqref{E:rank of n*n minor}.
\end{proof}
\begin{remark}
With the dimension formulas, we can compare the corank with the number \eqref{E:rank of n*n minor}. We observe that they appear to be always the same. But for the purpose of our project, we only need the inequality.  

For example , when $a=0, k_\bullet =5,6,\dots,14$ and $n=3$, we have 

\begin{center}
	\begin{tabular}{|c|c|c|c|c|c|c|c|c|c|c|c|}
		\hline  &
		$k_\bullet$ =    5 & 6 & 7 & 8 & 9 & 10 & 11 & 12&13&14&
		\\ \hline
		Corank of $P_3(4k_\bullet-2,0)$&1&2&1&2&1&1&1&1&0&1&
		\\ \hline
	\end{tabular}
\end{center}
\medskip
\quad  The numbers $\min \{3-d_{k,0}^\unr, d_{k,0}^\Iw- d_{k,0}^\unr-3\}$ in \eqref{E:rank of n*n minor} are:\\
\begin{center}
	\begin{tabular}{|c|c|c|c|c|c|}
		\hline $k_\bullet$&
		$d_{k,0}^\unr$& $d_{k,0}^\Iw$ & $3-d_{k,0}^\unr$ & $d_{k,0}^\Iw - d_{k,0}^\unr-3$& $\min \{3-d_{k,0}^\unr$ $d_{k,0}^\Iw- d_{k,0}^\unr-3\}$
		\\ \hline
		5&1&5&2&1&1
		\\ \hline
		6&1&6&2&2&2
		\\ \hline
		7&2&7&1&2&1
		\\ \hline
		8&1&8&2&4&2
		\\ \hline
		9&2&9&1&4&1
		\\ \hline
		10&2&10&1&5&1
		\\ \hline
		11&2&11&1&6&1
		\\ \hline
		12&2&12&1&7&1
		\\ \hline
		13&3&13&0&7&0
		\\ \hline
		14&2&14&1&9&1
		\\ \hline
	\end{tabular}
\end{center}
\end{remark}
\newpage
\section{Ghost conjecture and its application}
\subsection{Ghost conjecture}
In this section, we will state a version of the ghost conjecture and one useful theorem that can be used in the proof of the ghost conjecture. 

As in \cite{liu-wan-xiao}, we can define the characteristic power series of $U_5$ on $S_{k,a}^{D,\dagger}(\widehat\Gamma_1(5))$ as $$\textrm{Char}(U_5,k,a) =  \textrm{det}(I-X P(k,a)).$$ 

Our goal is to compute the 5-adic valuation of the $U_5$-eigenvalue, namely \emph{the slope}, by finding the Newton polygon of the power series $\textrm{Char}(U_5,k,a)$. The ghost conjecture states that there is an explicitly defined power series that has the same Newton polygon.\\ 

Set $w_k = \exp(5(k-2))-1$. We can find power series in $w$ such that when evaluating at $w=w_k$, we get the entries of the matrix. Note that the power series coefficients are not in $\ZZ_5$  but as in \cite{liu-wan-xiao}, there is a power series  $\textrm{Char}(U_5,w,t) \in \ZZ_5 [[w]][[t]]$ such that  ${Char}(U_5,w_k,a)={Char}(U_5,k,a)$ for all weights $k$.

Following \cite{bergdall-pollack}, we define  the \emph{ghost series} for each $a = 0,\dots,3$  to be the formal power series
\begin{equation}
	\label{ghost series}
	G^{\left(a\right)}(w,t)=  1+\sum_{n=1}^\infty
	g^{\left(a\right)}_n(w)t^n\in \ZZ_5 [w][[t]],
\end{equation}
where each coefficient $g^{\left(a\right)}_n(w)$ is a product
\begin{equation}
	\nonumber
	g^{\left(a\right)}_n(w) = \prod_{\substack{l \geq 2\\ l \equiv 2a-2\bmod 4 }} (w - w_\ell)^{m^{\left(a\right)}_n(\ell)} \in \ZZ_5[w]
\end{equation}
with exponents $m^{\left(a\right)}_n(\ell)$ given by the following recipe
\[
m^{\left(a\right)}_n(\ell) = \begin{cases}
	\min\big\{ n - d_{\ell,a}^\unr , d_{\ell,a}^\Iw- d_{\ell,a}^\unr - n\big\} & \textrm{ if }d_{\ell,a}^\unr < n < d_{\ell,a}^\Iw - d_{\ell,a}^\unr
	\\
	0 & \textrm{ otherwise.}
\end{cases}
\] 

Note that the exponents $m^{\left(a\right)}_n(l)$ has the same pattern as the corank of $P_n(k,a)$. We expect that the ghost conjecture is correct in this case.
\begin{conjecture}
	\title{(Ghost conjecture)}
	$\textrm{Char}(U_5,w,a)$ and $G^{(a)}(w,t)$ have the same Newton polygon for every $w \in \mathfrak{m}_{\CC_5}$. 
\end{conjecture}
There are more technical difficulties to prove this conjecture. However, we believe the following theorem is one among many ingredients. 

\begin{theorem}
	The determinant of the upper left $n\times n$ minor $P_n(k,a)$ is divisible by $g^{\left(a\right)}_i(w_k)$.
\end{theorem}
\begin{proof} 
	As noted above, each coefficients of the matrix is a power series in $w$. Then we can find a power series $f_{a,n}(w) \in \ZZ_5[[w]]$ such that $f_{a,n}(w_\ell)= \textrm{det}P_n(\ell,a)$. By the conrank bound in \ref{E:rank of n*n minor}, $f_n(w,a)$ is divisible by  $(w - w_\ell)^{m^{\left(a\right)}_n(\ell)}$, so is $g^{\left(a\right)}_i(w_k)$.
\end{proof}
\subsection{Gouv\^ea conjecture}
In this section we show that the first $d_{k_0,a}^\unr$-slopes of the Newton polygon of the ghost series is less than $\frac{k+4}{6}$, this is slightly larger than the  Gouv\^ea's $\frac{k-1}{p+1}$ conjecture. Note that we shall suppress $a$ from the notation when it is not necessary.

First we define some useful terms. 

\begin{lemma}
	\label{L:extremal ks}
	For $k=4k_{\bullet}+ 2a -2 $, fix $n\in \ZZ_{ \geq 0}$.
	
	\begin{enumerate}
		\item There is a unique integer $k_\bullet$ such that  $n=\frac 12 d_{k+4,a}^\Iw$; it is
		$$k_\bullet = k_{\midd\bullet}(n): = 2 n+1.$$
		
		\item We denote the largest $k_\bullet \in \ZZ_{\geq 0}$ satisfies $n = d_{k,a}^\unr$ as $k_{{\max}\bullet}$ and it is:   $$k_{{\max}\bullet}(n)=6n+a+2$$
		
		\item 
		There exists $k_\bullet \in \ZZ_{\geq 0}$ such that  $n \geq d_{k,a}^\Iw-d_{k,a}^\unr$ for any $k < k_\bullet$. We denote it as $k_{{\min}\bullet}$  
		
		If we write $$\tilde k_{{\min}\bullet}(n)= 6n+4-a+5\delta_{n,a,5},$$ then $k_{{\min}\bullet}(n) = \lceil \tilde k_{{\min}\bullet} (n)/ 5\rceil.$
	\end{enumerate}
	
\end{lemma}

\begin{proof}
	(1) This follow from the dimension formula in Theorem \ref{dim_iw} $d_{k,a}^\Iw = k_\bullet.$\\
	(2) This follow from the formula in Theorem \ref{dim_unr} 	$d_{k,a}^\unr = \lfloor \frac{k_\bullet-a+5}{6}\rfloor  -\delta_{k_\bullet,a+2,6}$.\\ 
	(3)	We see that $$d_{k,a}^\Iw-d_{k,a}^\unr=k_\bullet-\lfloor \frac{k_\bullet-a+5}{6}\rfloor -\delta_{k_\bullet,a+2,6}=k_\bullet- \lfloor\frac{k_\bullet-a+5-6\delta_{k_\bullet,a+2,6}}{6}\rfloor$$
	Hence,the equality $n \geq d_{k,a}^\Iw-d_{k,a}^\unr$ is equivalent to 
	
	$$ 6(k_\bullet-n) \leq k_\bullet-a+5-6\delta_{k_\bullet,a+2,6}$$
	
	We consider two cases. 
	
	\begin{itemize}
		\item Case 1: If $k_\bullet \not \equiv a+2 \bmod 6$, then $\delta_{k_\bullet,a+2,6} =0$. We have $$5k_\bullet \leq 6n-a+5$$
		
		If  $n \not\equiv a\bmod 5$, then   $6n-a+5$ is also not divisible by 5. Since $5k_\bullet \not \equiv -a+4 \bmod6$, we conclude that $ 5k_\bullet < 6n-a+4.$
		
		If $n \equiv a\bmod 5$, then $k_\bullet \leq \frac{6n-a+5}5$
		\item Case 2: If $k_\bullet \equiv a+2 \bmod 6$, then $\delta_{k_\bullet,a+2,6} =1$. In this case, $ 5k_\bullet \leq 6n-a-1.$
	\end{itemize}
	
	In both case, we set $$\tilde k_{{\min}\bullet}(n)= 6n+4-a+5\delta_{n,a,5},$$ then we conclude that for all $k_\bullet < \lceil \frac{k_{{\min}\bullet}(n)}5\rceil$, $n \geq d_{k,a}^\Iw-d_{k,a}^\unr$
\end{proof}

\begin{remark}
	\label{R:meaning of kmidminmax} 
	\begin{itemize}	From Lemma \ref{L:extremal ks} and the dimension formula in Theorem \ref{dim_iw} and Theorem \ref{dim_unr}, we see that
		\item $\frac 12d_{k,a}^\Iw = n \Leftrightarrow k_\bullet = k_{\mathrm{mid}\bullet}(n)$.
		\item $d_k^\unr(\varepsilon_1) \leq  n \Leftrightarrow k_\bullet \leq k_{\mathrm{max}\bullet}(n)$ and $k \not =  k_{\mathrm{max}\bullet}(n)-1$
		\item 
		$d_{k,a}^\Iw - d_{k,a}^\unr\leq  n \Leftrightarrow k_\bullet < k_{\mathrm{min}\bullet}(n)$ and $d_{k,a}^\Iw -d_{k,a}^\unr \geq   n \Leftrightarrow k_\bullet \geq k_{\mathrm{min}\bullet}(n-1)$.
	\end{itemize}
\end{remark}	
Now we can computes the difference between the exponent $m_n(\ell)$.

\begin{lemma}\label{ghostpower} For $k=4k_{\bullet}+ 2a -2 $, we have 
	\[m_{n+1}(k)-m_n(k) = \begin{cases}
		-1 & \textrm{ if } k_{\min\bullet}(n)\leq k_\bullet<k_{\midd\bullet}(n)
		\\
		0 & \textrm{ if } k_\bullet = k_{\min\bullet}(n) \textrm{ or } k_{\max\bullet}(n)-1
		\\
		1 & \textrm {if } k_{\min\bullet}(n)<k_\bullet<k_{\midd\bullet}(n) \textrm{ and  } k \not = k_{\max\bullet}(n)-1.
	\end{cases}
	\]
\end{lemma}

\begin{proof}
	\begin{itemize}
		For a fixed integer $n$ and consider $k$ satisfies $d_{k,a}^\unr<n<d_{k,a}^\Iw- d_{k,a}^\unr$, we have
		\[\min\big\{ n - d_{k,a}^\unr , d_{k,a}^\Iw- d_{k,a}^\unr - n\big\} =
		\begin{cases}
			n - d_{k,a}^\unr & \textrm{ if } n \leq \frac{d_{k,a}^\Iw}2
			\\
			d_{k,a}^\Iw- d_{k,a}^\unr - n & \textrm {otherwise.}
		\end{cases}\]
		From Remark \ref{R:meaning of kmidminmax}, we list the value of $m_{n+1}(k)$ and $m_n(k)$ for various $k$'s in the following table
		\begin{center}
			\begin{tabular}{|c|c|c|}
				\hline $k_\bullet$ &
				$m_{n+1}(k)$ & $m_{n}(k)$  
				\\ \hline
				$k_{\min\bullet}(n)\leq k_\bullet<k_{\midd\bullet}(n)$&$ d_{k,a}^\Iw- d_{k,a}^\unr - (n+1)$&$ d_{k,a}^\Iw- d_{k,a}^\unr - n$
				\\ \hline
				$k_\bullet = k_{\midd\bullet}(n)$&$d_{k,a}^\Iw- d_{k,a}^\unr - (n+1)$&$n-d_{k,a}^\unr$
				\\ \hline
				$k_\bullet = k_{\max\bullet}(n)-1$&$0$&$0$
				\\ \hline
				$k_{\midd\bullet}(n)< k_\bullet \leq k_{\max\bullet}(n),k_\bullet \not = k_{\max\bullet}(n)-1 $&$n+1-d_{k,a}^\unr$&$n-d_{k,a}^\unr$
				\\ \hline
				
			\end{tabular}
		\end{center}		
		
		The first case and the last case follows directly from the remark. We explain the other two cases. \\
		
		When $k_\bullet = k_{\midd\bullet}(n)$, from the dimension formula in Theorem \ref{dim_iw}, we have $n =\frac{d_{k+4,a}^\Iw}2=k_\bullet+1.$ Hence,
		$m^{(a)}_{n+1}(k) - m^{(a)}_{n}(k)= d_{k,a}^\Iw - 2n -1  =0$\\
		
		When  $k_\bullet = k_{\max\bullet}(n)-1$, we have $k_\bullet=6n+a+1$. By the formula in Theorem \ref{dim_unr}, we have $d_{k,a}^\unr=n+1$ so $m^{(a)}_{n+1} = 0,$ and since $d_{k,a}^\unr>n$,$m^{(a)}_{n} = 0.$
	\end{itemize}
\end{proof}
We also need the following useful formula. 

For a positive integer $m$, we denote $\Dig(m)$ as the sum of all digits in the $p$-adic expansion of $m$. For a real number $\alpha$, the number $\lceil \alpha \rceil$ is  the smallest integer that is not smaller than $\alpha$.
\begin{lemma}
	\label{E:sum of consecutive valuations}
	The sums of valuations of continuous integers in $(m_1, m_2]$ with $m_2 > m_1>0$ is $$\sum_{m_1< i \leq m_2} v_p(i) = \frac{(m_2-\Dig(m_2) )-(m_1 -\Dig(m_1))} {p-1}.$$
\end{lemma}
\begin{proof}
	We use the well-known formula $v_p(n!)= \frac{n-\Dig(n)}{p-1}$ to conclude that $$\sum_{m_1< i \leq m_2} v_p(i) = v_p(m_2!)-v_p(m_1!)= \frac{m_2-\Dig(m_2)}{p-1}-\frac{m_1 -\Dig(m_1)} {p-1}.$$
\end{proof}
\begin{lemma}
	\label{E:digitgame} 
	Let n,m be positive integers such that $pn > m$.  
	$$p\lceil \frac{m}{p} \rceil + \Dig(n - \lceil \frac{m}{p} \rceil) =  m + \Dig(pn -  m ).$$
\end{lemma}
\begin{proof}
	Write $m = p \lceil \frac{m}{p} \rceil - r$ where $r=0,\dots,p-1$. 
	We have $$   \Dig(n - \lceil \frac{m}{p} \rceil) =  \Dig(pn - p\lceil \frac{m}{p} \rceil) =\Dig(pn-m-r).$$
	Since $pn-m-r$ is divisible by $p$, the number $pn-m$ ends with digit $r$ in its $p$-based expression. Hence,  $$ \Dig(pn-m-r)= \Dig(pn-m)-r.$$	
\end{proof}
Now we are ready to find an upper bound for the slope. 
\begin{theorem}
	\label{P:gouvea k-1/p+1 conjecture}
	For an integer $k_0 \equiv 2a-2 \bmod{4}$, the first $d_{k_0,a}^\unr$-slopes of the Newton polygon of $G^{(a)}(w_{k_0},t))$ are all less than or equal to $\frac{k_0+4}{6}$. More precisely, they are less than or equal to the following.
	\begin{equation}
		\label{E:gouvea maximal slope}
		4(d_{k_0,a}^\unr-1+\delta_{k_{0\bullet},a+2,6})+a+1
	\end{equation}
\end{theorem}
\begin{proof}
	We first check that the number in $\eqref{E:gouvea maximal slope}$ is less than $\frac{k_0+4}{6}$.
	
	From Theorem \ref{dim_unr}, we have:
	$$4(d_{k_0,a}^\unr-1+\delta_{k_{0\bullet},a+2,6})+a+1= 4\lfloor\frac{k_{0\bullet}-a-1}{6}\rfloor+a +1\leq \frac{4k_{0\bullet}+2a+2}{6}=\frac{k_0+4}{6}$$
	
	Now, we prove that the first $d_{k_0}^\unr$-slopes of $\NP(G(w_{k_0}, -))$ are all less than or equal to \eqref{E:gouvea maximal slope}. It suffices to show that for $i = 1, \dots, d_{k_0}^\unr$, we have
	$$
	v_5(g_{d_{k_0}^\unr}(w_{k_0})) - v_5(g_{d_{k_0}^\unr-i}(w_{k_0})) \leq i\cdot \left(4(d_{k_0,a}^\unr-1+\delta_{k_{0\bullet},a+2,6})+a+1\right).
	$$
	We separate three cases $k_0 \equiv a+1 \bmod 6$, $k_0 \equiv a+2 \bmod 6$ and $k_0 \not \equiv a+1, a+2 \bmod 6$
	
	First, let consider the case $k_0 \not \equiv a+1, a+2 \bmod 6$. In this case the upper bound is  $4(d_{k_0,a}^\unr-1)+a+1$.
	
	For integers $n = d_{k_0}^\unr -1,\dots,1, 0$,
	we can check that $k_{0\bullet} > k_{{\max}\bullet}(n).$ So we have 
	
	\begin{eqnarray}
		\label{ghostvertex_diff}
		\nonumber
		& v_5(g_{n+1}(w_{k_0})) - v_5(g_n(w_{k_0})) \ = \hspace{-15pt}\displaystyle\sum_{\substack{
				k_\bullet \neq k_{\max\bullet(n)}-1
				\\
				k_{\midd\bullet}(n) < k_\bullet \leq k_{{\max}\bullet}(n)}}	
		\hspace{-20pt}\big( v_5(k_{0\bullet} - k_\bullet)+1 \big)\ \,-  \hspace{-25pt} \displaystyle\sum_{
			k_{{\min}\bullet}(n) \leq k_\bullet < k_{{\midd}\bullet}(n)}
		\hspace{-20pt}\big( v_5(k_{0\bullet}-k_\bullet)+1 \big)
		\\
		\nonumber
		&\qquad\qquad 
		\leq \dfrac{5(k_{{\max}\bullet}(n) - k_{\midd\bullet}(n))-4 + \Dig(k_{0\bullet} - k_{{\max}\bullet}(n)-1) - \Dig (k_{0\bullet}-k_{\midd\bullet}(n)-1)}{4}
		\\
		\nonumber
		& \qquad\qquad\qquad\quad -\ \dfrac{5(k_{{\midd}\bullet}(n) - k_{{\min}\bullet}(n))+ \Dig ( k_{0\bullet}- k_{\midd\bullet}(n)) - \Dig(k_{0\bullet} - k_{{\min}\bullet}(n)) 
		}{4}.
	\end{eqnarray}
	\quad Note that in the inequality, we add an extra term $v_5\left(k_\bullet - k_{\max\bullet(n)}+1\right)$ in order to use the following Lemma \ref{E:sum of consecutive valuations}. We will subtract it when it is necessary later. \\
	Using Lemma \ref{E:digitgame} and Lemma \ref{L:extremal ks}, one can deduce that
	\begin{equation}
		\label{E:digital expansion game}
		\nonumber
		5k_{{\min}\bullet}(n)+ \Dig(k_{0\bullet} - k_{{\min}\bullet}(n)) = \tilde k_{{\min}\bullet}(n)+ \Dig(5k_{0\bullet} - \tilde k_{{\min}\bullet}(n)).
	\end{equation}
	Hence, the expression is equal to \begin{align*}
		&\frac{5k_{{\max}\bullet}(n)+\tilde k_{{\min}\bullet}(n) - 10k_{{\midd}\bullet}(n)-4}{4} + \frac{\Dig(k_{0\bullet}- k_{{\max}\bullet}(n)-1) + \Dig(k_{0\bullet}  - \tilde k_{{\min}\bullet}(n))}{4}
		\\
		& \quad - \frac{\Dig(k_{0\bullet}- k_{{\midd}\bullet}(n)-1)+\Dig(k_{0\bullet}- k_{{\midd}\bullet}(n))}{4}
	\end{align*}
	We divide the sum into two parts: $$D_1(n)=\frac{5k_{{\max}\bullet}(n)+\tilde k_{{\min}\bullet}(n) - 10k_{{\midd}\bullet}(n)-4}{4}$$ and $$D_2(n)=\frac{\Dig(k_{0\bullet}  - \tilde k_{{\min}\bullet}(n)) - \Dig(k_{0\bullet}- k_{{\midd}\bullet}(n)-1)+\Dig(k_{0\bullet}- k_{{\midd}\bullet}(n))}{4}.$$
	From Lemma ~\ref{L:extremal ks}, we have
	\begin{equation}\label{D1}
		D_1(n)= \frac{5(6n+a+2)+(6n+4-a+\delta_{n,a,5})-10(2n+1)-4}{4}=4n+a+\frac{\delta_{n,a,5}}{4}
	\end{equation}
	Now we simplifying the second term $D_2(n)$. \\
	We have $$\Dig\left(5k_{0\bullet} - \tilde k_{{\min}\bullet}(n)\right) = \Dig\left(5k_{0\bullet}-(6n+4-a+\delta_{n,a,5})\right)$$
	If $n \equiv a \bmod 5$ then it is equal to $	\Dig\left(5k_{0\bullet}-6n-9+a\right).$We observe that $5k_{0\bullet}-6n-4+a$ end with digit 1 when written in 5-base.
	Hence, $$\Dig\left(5k_{0\bullet}-6n-9+a\right) = \Dig\left(5k_{0\bullet}-(6n+4-a)-2\right)-3$$   
	
	If $n \not \equiv a \bmod 5$, we have  $$\Dig\left(5k_{0\bullet}-(6n+4-a+\delta_{n,a,5})\right)=\Dig(5k_{0\bullet}-6n-4+a) \leq \Dig\left(5k_{0\bullet}-6n-6+a\right)+2 $$
	
	Hence, $$\Dig\left(5k_{0\bullet} - \tilde k_{{\min}\bullet}(n)\right)\leq \Dig\left(5k_{0\bullet}-6n-6+a\right)+(2-\delta_{n,a,5})$$
	We write	$$A_n = k_{0\bullet}  - k_{{\max}\bullet}(n)-1= k_{0\bullet}-6n-a-3,$$ 
	$$B_n = 5k_{0\bullet}-6n-6+a,$$
	$$C_n = k_{0\bullet}- k_{{\midd}\bullet}(n)= k_{0\bullet}-2n-2.$$
	
	From the argument above, we conclude
	$$D_2(n) \leq\frac{\Dig\left(A_n\right)+\Dig\left(B_n\right)-\Dig\left(C_n\right)-\Dig\left(C_n+1\right)+(2-\delta_{n,a,5})}{4}$$
	
	Combine with \eqref{D1}, we have 
	$$D_1(n)+D_2(n) \leq 4n+a+\frac 12 +  \frac{\Dig\left(A_n\right)+\Dig\left(B_n\right)-\Dig\left(C_n\right)-\Dig\left(C_n+1\right)}{4}$$
	
	Thus, we have:
	\begin{eqnarray}
		\label{F}
		&v_p(g_{d_{k_0}^\unr}(w_{k_0})) - v_p(g_{d_{k_0}^\unr-i}(w_{k_0})) -  i \cdot \big(4(d_{k_0,a}^\unr-1)+a+1\big)  \\
		\nonumber &
		\leq i\left(\frac{3}{2}-2i\right)+\displaystyle \sum_{1 \leq j \leq i}   \frac{\Dig\left(A_n\right)+\Dig\left(B_n\right)-\Dig\left(C_n\right)-\Dig\left(C_n+1\right)}{4}
	\end{eqnarray}
	
	We find an upper bound for the last term. 
	\begin{lemma}
		\label{ABC}
		$\displaystyle \sum_{1 \leq j \leq i} \Dig\left(A_n\right)+\Dig\left(B_n\right)-\Dig\left(C_n\right)-\Dig\left(C_n+1\right) \leq 4i^2+2i$
	\end{lemma}
	We have the following relations:
	$$B_n = 6 C_n - A_n+3$$. 
	
	For $j=1,2,...,d^\unr_{k_0}$, we write $A_{d_{k_0}^\unr - j},B_{d_{k_0}^\unr - j},C_{d_{k_0}^\unr - j}$ in terms of $A= A_{d_{k_0}^\unr - 1}$,and $C= C_{d_{k_0}^\unr - 1}$.
	
	Note that 	$A_{d_{k_0}^\unr-1}= k_{0\bullet}-6(d_{k_0}^\unr-1)-a-3=k_{0\bullet}-a-1-\lfloor \frac{k_{0\bullet}-a-1}6\rfloor-1 \in \{0,1,2,3\}.$ 
	
	\begin{equation}
		\label{E:A_n and C_n}
		\nonumber
		A_{d_{k_0}^\unr - j} = A+ 6(j-1) 
		\textrm {    and  }C_{d_{k_0}^\unr - j} = C+2(j-1)
	\end{equation}
	
	\begin{equation}
		\nonumber
		\textrm{ Hence, } B_{d_{k_0}^\unr - j} = 6(C+j-1)+3-A
	\end{equation}
	Then using the inequality $\textrm{Dig}(A+B)\leq\textrm{Dig}(A) + \textrm{Dig}(B)  $ and $\textrm{Dig}(6k)\leq 2\textrm{Dig}(k) \leq 2k$, we conclude that:\\
	\begin{equation}
		\nonumber
		\textrm{Dig}\big(A_{d_{k_0}^\unr - j}\big) \leq \Dig(A)+2(j-1)= A+2(j-1)
	\end{equation}
	
	\begin{equation}
		\label{Eq1}
		\nonumber
		\textrm{Dig}\big(B_{d_{k_0}^\unr - j}\big) \leq \Dig\left(6(C+j-1)\right)+\Dig(3-A)\leq \Dig\left(6(C+j-1)\right)+3-A
	\end{equation}  
	
	Adding these inequalities:
	
	\begin{align*}
		\displaystyle \sum_{1 \leq j \leq i}{\big( \textrm{Dig} A_{d_{k_0}^\unr - j} +\textrm{Dig}B_{d_{k_0}^\unr-j}}\big) \leq i^2+ 2i+\displaystyle \sum_{1 \leq j \leq i}{\textrm{Dig} \big(6(C+j-1)\big)}.
	\end{align*}
	
	We need the following lemma
	
	\begin{lemma}
		\label{Diglemma}
		For all $k,C \in \ZZ_{+}$
		$$ \textrm{Dig}\big((p+1)C \big) \leq \textrm{Dig}(C)+ \textrm{Dig}(C+k)+k(p-2) $$ 
	\end{lemma}
	\begin{proof}
		We prove by induction on the number of the digit of $C$ written in $p$-base.\\ 
		If C has only one digit, we have $$\textrm{Dig}(C)+ \textrm{Dig}(C+k)+k(p-2) \geq \textrm{Dig}(C)+p-1 \geq  2 \textrm{Dig}(C) = 2 C = \textrm{Dig}\big((p+1)C\big).$$ 
		Now we assume that the inequality is true for any $C$ with at most $n-1$ digits, and prove for the case C has $n$ digits.

		If $k$ has at least $n$ digits then $k(p-2) \geq p^{n-1}(p-2) \geq (n-1)p(p-2) \geq (n+1)(p-1)$. The last inequality holds for $n \geq 2$, and $p \geq 5$. We observe that either (p+1)C has $n+1$ digits or it has (n+2) digits and start with 10, so in both case the sum of its digit is at most $(n+1)(p-1)$. The lemma follows. \\
		Now we assume $k$ has $r\leq n-1$ digits and consider 2 cases. \\
		\emph{Case 1:} If the first $n-r$ digits of C are all equal $p-1$, we write $C=p^r D + E$ where $0 \leq E < p^r.$ \\
		If $E+k < p^r$, then $\textrm{Dig}(C+k)= \textrm{Dig}(D)+\textrm{Dig}(E+k)$. By the inductive hypothesis on E, we deduce 
		\begin{eqnarray}
			\nonumber
			& \textrm{Dig}(C) + \textrm{Dig}(C+k) + k(p-2)= 2 \textrm{Dig}(D)+\textrm{Dig}(E)+ \textrm{Dig}(E+k)+k(p-2) \geq \\
			\nonumber & \textrm{Dig}\big((p+1)D\big)+ \textrm{Dig}\big((p+1)E\big) \geq 
			\textrm{Dig}\big((p+1)C \big)
		\end{eqnarray}
		If $E+k\geq p^r$, we can compute $\textrm{Dig}(C+k)= \textrm{Dig}(E+k), \textrm{Dig}(C)= (n-r)(p-1)+\textrm{Dig}(E)$ and $\textrm{Dig}\big((p+1)C\big) \leq (n-r)(p-1)+\textrm{Dig}\big((p+1)E\big)$. The lemma follows from the inductive hypothesis on E. 
		
		\emph{Case 2:} $C$ has a digit other than $p-1$ in the first $n-r$ digit, let say it happen at $p^{r'}$. We write $C = p^{r'-1} D+E$, where  $0 \leq E < p^{r'-1}$.  In this case either $E+k \geq p^{r'}$ or $E+k < p^{r'}$ , we have $\textrm{Dig}(C+k) = \textrm{Dig}(D)+ \textrm{Dig}(E+k)$. This is the same situation as in Case 1. 
		
	\end{proof}
	Applying the lemma for $C,C+1,...,C +i-1$, $k=i$ and $p=5$, then summing up the inequalities, we deduce
	
	$$\displaystyle \sum_{1 \leq j \leq i}{\Dig\left(6(C+j)\right)} \leq \displaystyle \sum_{1 \leq j \leq 2i}{\textrm{Dig}\big((C+j-1)\big)}+3i^2).$$
	
	Hence,
	
	\begin{equation}
		\label{bound1}
		\displaystyle \sum_{1 \leq j \leq i} 
		\textrm{Dig} \left( A_{d_{k_0}^\unr - j} \right) +\textrm{Dig} \left( B_{d_{k_0}^\unr - j}\right) - 2 \textrm{Dig} \left(C_{d_{k_0}^\unr - j} \leq 4i^2+2i \right)	\end{equation}
	
	Thus, the inequality \ref{F} becomes 
	
	$$	v_5(g_{d_{k_0}^\unr}(w_{k_0})) - v_p(g_{d_{k_0}^\unr-i}(w_{k_0})) -  i \cdot \big(4(d_{k_0,a}^\unr-1)+a\big) \leq i(2-i).$$
	
	We only need to check for the case $i=1$.
	
	For $i=1$, using the notation in Lemma \ref{ABC}
	\begin{eqnarray} 
		\nonumber
		&v_5(g_{d_{k_0}^\unr}(w_{k_0})) - v_5(g_{d_{k_0}^\unr-1}(w_{k_0})) -  \big(4(d_{k_0}^\unr-1)+a+1\big) 
		\\
		\nonumber
		& \leq -\frac 12 + \frac{\Dig(A)+\Dig(6C+3-A)-\Dig(C)-\Dig(C+1)}4
	\end{eqnarray}
	\begin{itemize}
		\item If $A = 0,1, \textrm{ or }2$, then we have  $$\Dig(6C+3-A)\leq  \Dig(5C)+\Dig(C+1)+\Dig(2-A)= 2-A+\Dig(C)+\Dig(C+1)$$
		\item	If $A=3$, recall that in \ref{ghostvertex_diff}, we put the extra term $$v_5\left(k_\bullet - k_{\max(n)\bullet}+1\right)=v_5(A+2)=1.$$
		
	\end{itemize}

	So in both cases, $v_5(g_{d_{k_0}^\unr}(w_{k_0})) - v_5(g_{d_{k_0}^\unr-1}(w_{k_0})) -  \big(4(d_{k_0}^\unr-1)+a+1\big) \leq 0.$
	
	We finish the proof for the case $k_0 \not \equiv a+1, a+2 \bmod 6$.
	
	Let us consider $k_{0\bullet}\equiv a+1 \bmod 6$. In this case $k_{0\bullet}> k_{\max(n)\bullet}$ for  $n = d_{k_0}^\unr -2,\dots,1, 0$,and  $k_{0\bullet}= k_{\max(n)\bullet}-1$  when $n = d_{k_0}^\unr-1.$
	So the computation still works for the former. We just need to modify the proof for $n = d_{k_0}^\unr-1$. In this case we can compute 
	\begin{eqnarray}
		\label{a+1_upperbound}
		& v_5(g_{n+1}(w_{k_0})) - v_5(g_n(w_{k_0})) \ = \hspace{-15pt}\displaystyle\sum_{\substack{
				k_\bullet \neq k_{\max(n)\bullet}-1
				\\
				\nonumber
				k_{\midd\bullet}(n) < k_\bullet \leq k_{{\max}\bullet}(n)}}	
		\hspace{-20pt}\big( v_5(k_{0\bullet} - k_\bullet)+1 \big)\ \,-  \hspace{-25pt} \displaystyle\sum_{
			k_{{\min}\bullet}(n) \leq k_\bullet < k_{{\midd}\bullet}(n)}
		\hspace{-20pt}\big( v_5(k_{0\bullet}-k_\bullet)+1 \big)
		\\
		\nonumber
		&\qquad\qquad 
		= \dfrac{5(k_{{\max}\bullet}(n) - k_{\midd\bullet}(n))-6  - \Dig (k_{0\bullet}-k_{\midd\bullet}(n)-1)}{4}
		\\
		\nonumber
		& \qquad\qquad\quad -\ \dfrac{5(k_{{\midd}\bullet}(n) - k_{{\min}\bullet}(n))+ \Dig ( k_{0\bullet}- k_{\midd\bullet}(n)) - \Dig(k_{0\bullet} - k_{{\min}\bullet}(n)) 
		}{4}.
	\end{eqnarray}
	If we define  $\Dig(k_{0\bullet}  - k_{{\max}\bullet}(n)-1)= \Dig(-2) = \ -2$, we have the upper bound in \ref{a+1_upperbound} is the same as in \ref{ghostvertex_diff}. Hence, we can apply the exact argument for $A_{d_{k_0}^\unr-1}=-2 $ as in the previous case where $A_{d_{k_0}^\unr-1}=0,1,2,3.$
	
	We left with the case $k_{0\bullet}\equiv a+1 \bmod 6$. In this case we have 
	$$\Dig\left(A_{d_{k_0}^\unr - j}\right)=\Dig\left(5 + 6(j-1)\right)\leq \Dig(5)+\Dig(6(j-1)) \leq 2j-1$$
	
	Let $C= C_{d_{k_0}^\unr}$, then we can write  $C_{d_{k_0}^\unr - j}= C+2j.$
	
	Thus, 
	$$\Dig\left(B_{d_{k_0}^\unr - j}\right)=\Dig\left(6C_{d_{k_0}^\unr -j}- A_{d_{k_0}^\unr - j}+3 \right) \leq \Dig(6(C+j))+4$$
	
	We have $$\Dig \left( A_{d_{k_0}^\unr - j}\right)+\Dig \left( B_{d_{k_0}^\unr - j} \right)\leq 2j+3 + \Dig\left(6(C+j)\right)$$
	
	The bound in Lemma \ref{ABC} in this case is $4i^2+4i$.
	
	So 	\begin{eqnarray}
		&v_p(g_{d_{k_0}^\unr}(w_{k_0})) - v_p(g_{d_{k_0}^\unr-i}(w_{k_0})) -  i \cdot \big(4d_{k_0,a}^\unr+a+1\big)  \\
		\nonumber &
		\leq i\left(-\frac{5}{2}-2i\right)+\displaystyle \sum_{1 \leq j \leq i}   \frac{\Dig\left(A_n\right)+\Dig\left(B_n\right)-\Dig\left(C_n\right)-\Dig\left(C_n+1\right)}{4} \leq i(-\frac{3}{2}-i) <0.
	\end{eqnarray}
	This completes the proof of Theorem \ref{P:gouvea k-1/p+1 conjecture}.
\end{proof}

\end{document}